\documentclass{amsart}

\usepackage[foot]{amsaddr}

\usepackage[lite,initials,nobysame,non-compressed-cites]{amsrefs}
\usepackage{amssymb}
\usepackage{hyperref}

\numberwithin{equation}{section}

\theoremstyle{plain} 
\newtheorem{thm}{Theorem}[section]
\newtheorem{cor}[thm]{Corollary}
\newtheorem{lem}[thm]{Lemma}

\theoremstyle{definition}

\newtheorem{qn}[thm]{Question}
\newtheorem*{rem}{Remark}

\renewcommand\le{\leqslant}
\renewcommand\ge{\geqslant}

\usepackage{stmaryrd}
\usepackage{xcolor}

\usepackage{array,longtable}
\newcolumntype{L}{>{$}l<{$}}
\newcolumntype{C}{>{$}c<{$}}

\begin{document}

\title[Abstract regular polytopes of $G_2(q)$]{Abstract regular polytopes of $G_2(q)$ \\ in even characteristic} 


\author[Malcolm H. W. Chen]{Malcolm Hoong Wai Chen} \address{Department of Mathematics, University of Manchester, Manchester M13 9PL, UK} \email{malcolmhoongwai.chen@manchester.ac.uk}


\begin{abstract}
For every even prime power $q>2$, we explicitly construct a rank five string C-group representation of the Chevalley simple group $G_2(q)$ using its natural seven-dimensional matrix representation. We also give a census for the string C-group representations of $G_2(q)$ of ranks four and above for every prime power $3 \le q \le 11$; providing computational evidence that rank five string C-group representations of $G_2(q)$ are rare and that this construction is exceptional among the currently known examples.
\end{abstract}

\maketitle

\section{Introduction}

Let $G$ be a group and $S=\{s_1,\dots,s_r\}$ be a set of involutions generating $G$ such that $s_i$ and $s_j$ commute whenever $|i-j| \ge 2$. Then, $(G,S)$ is called a \emph{string group generated by involutions} (or \emph{SGGI} for brevity), and $r$ is the \emph{rank} of $(G,S)$. A \emph{string C-group} is a SGGI $(G,S)$ such that $\langle s_i : i \in I \rangle \cap \langle s_j : j \in J \rangle = \langle s_k : k \in I \cap J \rangle$ for every $I,J \subseteq \{1,\dots,r\}$; here $S$ is called a \emph{C-string} and its \emph{Schl\"{a}fli type} is the sequence $[p_1,\dots,p_{r-1}]$ where $p_i=|s_is_{i+1}|$ for every $1 \le i \le r-1$. As documented in McMullen and Schulte's monograph \cite[Section 2E]{McMullenSchulte-arpbible} there is a one-to-one correspondence between C-strings of $G$ and abstract regular polytopes whose automorphism group is $G$. As such, investigations on string C-groups are of interest not only in group theory but also in incidence geometry. 

The quest to classify string C-groups started from early experimental results \cite{Hartley-atlas,LVauthier-atlas} which resulted in atlases of abstract regular polytopes for small groups. More efficient computer algorithms have since been developed to enumerate the C-strings of sporadic simple groups \cite{HartleyHulpke-algorithm,LM-algorithm,LMulpas-algorithm}. The possible ranks of C-strings have been investigated for some infinite families of (almost) simple groups, namely Suzuki groups ${}^{2}\mathrm{B}_2(q)$ \cite{L-Sz}, small Ree groups ${}^{2}\mathrm{G}_2(q)$ \cite{LSchulteVanMaldeghem-Ree}, linear groups of dimension at most four \cite{ConnorDeSaedeleerL-sl2socle,LSchulte-sl2,LSchulte-gl2,BVicinsky-gl3,BL-psl4}, orthogonal and symplectic groups \cite{BFerraraL-orthogonal,B-orthogonal}, as well as symmetric and alternating groups \cite{FL-symmetric,CFLM-alternating,FL-alternating}. A complete list of simple groups that do not admit any C-strings has been determined in \cite{LVandenschrick}. See the survey article by Leemans \cite{L-arpsurvey} for more details of these investigations, and Brooksbank \cite[Figures 1 \& 2]{B-orthogonal} for a summary on the possible ranks of these groups. 

Among the exceptional groups of Lie type, the Suzuki groups and the small Ree groups constitute the Lie rank one families, and they only admit rank three C-strings. Thus, the Chevalley simple groups $G_2(q)$, which are the smallest exceptional groups of Lie rank two, form the next natural family to investigate.

The main theorem of this paper is as follows.

\begin{thm} \label{main}
For every even prime power $q>2$, the Chevalley simple group $G_2(q)$ admits a C-string $\{t_1,t_2,t_3,t_4,t_5\}$ with Schl\"{a}fli type $[q+1,3,8,3]$
such that $$\langle t_1,t_2,t_3 \rangle \cong \mathrm{SL}_2(q), \ \ \langle t_1,t_2,t_3,t_4 \rangle \cong \mathrm{SL}_3(q):2, \ \text{and} \ \ \langle t_2,t_3,t_4,t_5 \rangle \cong G_2(2).$$
\end{thm}

The string C-group representation given in Theorem \ref{main} also allow us to construct C-string of lower ranks via a rank reduction theorem of Brooksbank and Leemans \cite{BL-rankreduction}. Thus, we obtain the following corollary.

\begin{cor} \label{notsomain}
For every even prime power $q>2$, the Chevalley simple group $G_2(q)$ admits a string C-group representation of ranks three and four with Schl\"{a}fli type $[8,2(q+1)]$ and $[12,6,q+1]$ respectively.
\end{cor}

\begin{rem}
The construction and proof presented here was completed independently prior to the author became aware of a very recent work of Gvozdev and Nuzhin \cite{GvozdevNuzhin-g2} which also established the existence of ranks three, four, and five string C-group representations of $G_2(q)$ for every even prime power $q>2$. In particular, the dual of their rank five C-string has Schl\"{a}fli type $[q+1,4,8,3]$ and is such that $\langle t_1,t_2,t_3 \rangle \cong q^2:D_{q+1}$. Likewise, the dual of their ranks three and four C-strings, which were obtained without appealing to the rank reduction theorem has Schl\"{a}fli type $[7,2(q+1)]$ and $[4,8,q+1]$ respectively. Consequently, their constructions is neither isomorphic nor dual to the constructions presented in this paper. We will compare and discuss both constructions in this paper.
\end{rem}

\subsection*{Outline} In Section \ref{prelim}, we recall the necessary background on string C-groups and a seven-dimensional matrix representation of $G_2(q)$. Then, in Section \ref{proof}, we use this representation to construct five involutions $t_1,\dots,t_5$ and show that $\{t_1,\dots,t_5\}$ is a rank five C-string of $G_2(q)$ when $q > 2$ is even, thereby proving Theorem \ref{main}. The key ingredient of the proof is to establish the following subgroup chain, with the involutions carefully chosen so as to ensure that the intersection condition holds.
\begin{equation*}
\begin{aligned}
&\underset{\mathrm{SL}_2(q)}
        {\langle t_1,t_2,t_3\rangle}
<
\underset{q^2:\mathrm{SL}_2(q)}
        {\langle t_1,t_2,t_3,(t_3t_4)^4\rangle}
<
\underset{\mathrm{SL}_3(q)}
        {\langle t_1,t_2,t_3,(t_3t_4)^2\rangle}
<
\underset{\mathrm{SL}_3(q):2}
        {\langle t_1,t_2,t_3,t_4\rangle}
<
\underset{G_2(q)}
        {\langle t_1,t_2,t_3,t_4,t_5\rangle}.
\end{aligned}
\end{equation*}
Finally, in Section \ref{conclusion}, we give computational census data for C-strings of $G_2(q)$ for every prime power $3 \le q \le 11$, which suggests that rank five C-strings of $G_2(q)$ are remarkably rare and belong to only two uniform families, namely the family constructed in this paper and that of Gvozdev and Nuzhin \cite{GvozdevNuzhin-g2}. We also discuss some directions for future research: whether five is the maximum rank in the even characteristic case and what observations can be drawn about the odd characteristic case. Throughout, we use group-theoretic notation following the ATLAS of Finite Groups \cite{atlas} and naming conventions returned by the \texttt{GroupName} command in \textsc{Magma} \cite{magma}. We also refer to the work of Bray, Holt, and Roney-Dougal \cite{BrayHoltRoneyDougal-maxbible} as a standard comprehensive reference for maximal subgroups of low-dimensional classical groups.

\section{Preliminaries} \label{prelim}

We begin by recalling a standard lemma which we will use throughout the paper to verify the intersection condition for string C-groups.

\begin{lem}{\cite[2E16(a) \& 11A10]{McMullenSchulte-arpbible}} \label{ip}
Let $(\Gamma,\{ s_1,\dots,s_r \})$ be a sggi where both $\{s_1,\dots,s_{r-1}\}$ and $\{s_2,\dots,s_r\}$ are C-strings. Then $\Gamma$ is a string C-group if either of the following conditions is satisfied.
\begin{enumerate}
\item $\langle s_1,\dots,s_{r-1} \rangle \cap \langle s_2,\dots,s_r \rangle = \langle s_2,\dots,s_{r-1} \rangle$.
\item $s_r \notin \langle s_1,\dots,s_{r-1} \rangle$, and $\langle s_2,\dots,s_{r-1} \rangle$ is maximal in $\langle s_2,\dots,s_r \rangle$. 
\end{enumerate}
\end{lem}

We give a natural description of $G_2(q)$ as a seven-dimensional matrix group, following Wilson's approach in \cite[Sections 4.3 \& 4.4]{Wilson-cfsgbible}. Recall that $G_2(q)$ is the automorphism group of an octonion algebra over $\mathbb{F}_q$, and hence may be realized as a matrix group via its action on the octonions.

Henceforth, fix an even prime power $q>2$. Working over characteristic two, we can choose a basis $\{x_1,\dots,x_8\}$ such that $x_4+x_5=1$ and $\langle x_4+x_5 \rangle^\perp = \langle x_1,x_2,x_3,x_4+x_5,x_6,x_7,x
_8 \rangle$. See Table \ref{octonion} for the multiplication table of this basis.

\begin{table}[h!]
\caption{Multiplication in the octonion basis $\{x_1,\dots,x_8\}$.}
\begin{tabular}{|c|cccccccc|} \hline 
& $x_1$ & $x_2$ & $x_3$ & $x_4$ & $x_5$ & $x_6$ & $x_7$ & $x_8$ \\ \hline
$x_1$ &  &  &  &  & $x_1$ & $x_2$ & $x_3$ & $x_4$ \\
$x_2$ &  &  & $x_1$ & $x_2$ &  &  & $x_5$ & $x_6$ \\
$x_3$ &  & $x_1$ &  & $x_3$ &  & $x_5$ &  & $x_7$ \\
$x_4$ & $x_1$ &  &  & $x_4$ &  & $x_6$ & $x_7$ &  \\
$x_5$ &  & $x_2$ & $x_3$ &  & $x_5$ &  &  & $x_8$ \\
$x_6$ & $x_2$ &  & $x_4$ &  & $x_6$ &  & $x_8$ &  \\
$x_7$ & $x_3$ & $x_4$ &  &  & $x_7$ & $x_8$ &  &  \\
$x_8$ & $x_5$ & $x_6$ & $x_7$ & $x_8$ &  &  &  &  \\ \hline
\end{tabular} \label{octonion}
\end{table}

The associated bilinear form $f$ satisfies $f(x_i,x_{9-i})=1$ for $i \in \{1,\dots,8\}$ and $f=0$ otherwise. Let $V$ be the vector space $\langle x_4+x_5\rangle^\perp$, and rewrite its basis as $\{y_1,\dots,y_7\}$ where
\begin{equation*}
y_i = 
\begin{cases}
x_i, & i=1,2,3, \\
x_4+x_5, & i=4, \\
x_{i+1}, & i=5,6,7.
\end{cases}
\end{equation*}

We regard vectors in $V$ as row vectors, so $G_2(q)$ acts on $V$ by right multiplication. Thus, we write $v^g$ for the image of $v \in V$ under the action of $g \in G_2(q)$. We also write $u \cdot v$ for the octonion multiplication of $u$ and $v$. Note that $f$ gives us a form completely defined on the basis vectors by
$f(y_i \cdot y_j , y_k) = 1$ if $$\{i,j,k\} \in \big\{ \{1,4,7\}, \{1,5,6\}, \{2,3,7\}, \{2,4,6\}, \{3,4,5\} \big\},$$
and $0$ otherwise. We can now define $$G = \{g \in \mathrm{GL}_7(q) : f(u^g \cdot v^g , w^g) = f(u \cdot v , w) \text{ for every $u,v,w \in V$} \}.$$
Then $G \cong G_2(q)$, and the following formula allows us to use the form $f$ to check whether a matrix in $\mathrm{GL}_7(q)$ also belongs to $G$. To reduce notation clutter, we write $(u,v)_{ij}$ for $u_iv_j+u_jv_i$, and note that $(u,v)_{ii}=0$ since we are in characteristic two.
\begin{lem} \label{form}
Let $u=\sum_{i=1}^7 u_i y_i$, $v=\sum_{i=1}^7 v_i y_i$, $w=\sum_{i=1}^7 w_i y_i \in V$. Then
\begin{align*}
f(u \cdot v , w) =& \big( (u,v)_{47} + (u,v)_{56} \big) w_1 \\
+& \big( (u,v)_{37} + (u,v)_{46} \big) w_2
\\
+& \big( (u,v)_{27} + (u,v)_{45} \big) w_3
\\
+& \big( (u,v)_{17} + (u,v)_{26} + (u,v)_{35} \big) w_4
\\
+& \big( (u,v)_{16} + (u,v)_{34} \big) w_5
\\
+& \big( (u,v)_{15} + (u,v)_{24} \big) w_6
\\
+& \big( (u,v)_{14} + (u,v)_{23} \big) w_7.
\end{align*}
\end{lem}

\begin{proof}
Let $I=\big\{ \{1,4,7\}, \{1,5,6\}, \{2,3,7\}, \{2,4,6\}, \{3,4,5\} \big\}$. Then
$$f(u \cdot v , w) = \sum_{i=1}^7\sum_{j=1}^7\sum_{k=1}^7 u_iv_jw_kf(y_i \cdot y_j , y_k) = \sum_{\{i,j,k\} \in I} u_i v_j w_k.$$
For each fixed $k$, we determine all possible pairs of $\{i,j\}$ in $I$ to determine the corresponding coefficient of $w_k$.
\end{proof}

We also make use of information on involution centralizers of $G$ given in \cite{scottellis}. Let \\
\begin{center}
$t=\begin{pmatrix} 1 & 0 & 0 & 1 & 0 & 0 & 1 \\ 0 & 1 & 0 & 0 & 1 & 0 & 0 \\ 0 & 0 & 1 & 0 & 0 & 1 & 0 \\ 0 & 0 & 0 & 1 & 0 & 0 & 0 \\ 0 & 0 & 0 & 0 & 1 & 0 & 0 \\ 0 & 0 & 0 & 0 & 0 & 1 & 0 \\ 0 & 0 & 0 & 0 & 0 & 0 & 1 \end{pmatrix}$, \\[0.2in]
$K = \left\{ \begin{pmatrix} 1 & 0 & 0 & a & b & c & a^2 \\ 0 & 1 & 0 & 0 & a & 0 & c \\ 0 & 0 & 1 & 0 & 0 & a & b \\ 0 & 0 & 0 & 1 & 0 & 0 & 0 \\ 0 & 0 & 0 & 0 & 1 & 0 & 0 \\ 0 & 0 & 0 & 0 & 0 & 1 & 0 \\ 0 & 0 & 0 & 0 & 0 & 0 & 1 \end{pmatrix} : a,b,c \in \mathbb{F}_q \right\}$, \\[0.2in]
$L = \left\{ \begin{pmatrix} 1 & 0 & 0 & 0 & 0 & 0 & 0 \\ 0 & \alpha & \beta & 0 & 0 & 0 & 0 \\ 0 & \gamma & \delta & 0 & 0 & 0 & 0 \\ 0 & 0 & 0 & 1 & 0 & 0 & 0 \\ 0 & 0 & 0 & 0 & \alpha & \beta & 0 \\ 0 & 0 & 0 & 0 & \gamma & \delta & 0 \\ 0 & 0 & 0 & 0 & 0 & 0 & 1 \end{pmatrix} : \alpha,\beta,\gamma,\delta \in \mathbb{F}_q \ , \ \alpha\delta+\beta\gamma=1 \right\}$. \\[0.2in]
\end{center}

In \cite[Lemmas 6.1 \& 6.2]{scottellis}, it was shown that $t \in G$, the subgroup $K$ is abelian, $K \cap L = 1$, and $C_G(t)=KL$. It is also immediate that $L \cong \mathrm{SL}_2(q)$, since any element $\ell \in L$ can be uniquely identified by the $2 \times 2$ block $$\widehat{\ell} = \begin{pmatrix} \alpha & \beta \\ \gamma & \delta \end{pmatrix},$$ and the condition $\alpha\delta+\beta\gamma=1$ is equivalent to $\det(\widehat{\ell})=1$ in characteristic two. We also write $k(a,b,c)$ to denote the element of $K$ corresponding to the parameters $a,b,c$. Since $(a+a')^2=a^2+(a')^2$ in characteristic two, we have
\begin{equation*}
k(a,b,c) k(a',b',c') = k(a+a',b+b',c+c').
\end{equation*}
Thus, every non-identity element of $K$ has order two, and we can also identify $K$ as an $\mathbb{F}_q$-vector space. We now prove a lemma describing its $\mathbb{F}_qL$-module structure.

\begin{lem} \label{mod}
The conjugation action of $L$ on $K$ is given by $$k(a,b,c)^\ell := \ell^{-1} k(a,b,c)\ell = k(a,b\alpha+c\gamma,b\beta+c\delta).$$
Consequently, $C_G(t)=K \rtimes L$. Moreover, $$K_a=\{k(a,0,0) : a \in \mathbb{F}_q\} \ \text{ and } \ K_{bc}=\{k(0,b,c):b,c\in\mathbb F_q\}$$ are $L$-invariant subgroups of $K$. Furthermore, $K=K_a \oplus K_{bc}$ as an $\mathbb{F}_qL$-module and $K_{bc}$ is an irreducible $\mathbb{F}_qL$-module. Finally, if $\pi : K_{bc}L \to L$ is the natural projection map and $H \le K_{bc}L$ satisfies $\pi(H)=L$, then $H \cap K_{bc}=1 \text{ or } K_{bc}$.
\end{lem}

\begin{proof}
The formula for the action of $L$ on $K$ follows from direct calculation. In particular, it implies that $K$ is normalized by $L$. Since $C_G(t)=KL$ and $K \cap L = 1$ by \cite[Lemma 6.2]{scottellis}, we have $C_G(t) = K \rtimes L$. It is also immediate from the action that $K_a$ and $K_{bc}$ are $L$-invariant subgroups of $K$. Since every element $k(a,b,c) \in K$ can be uniquely written as $k(a,b,c)=k(a,0,0)k(0,b,c) \in K_aK_{bc}$, it follows that $K=K_a \oplus K_{bc}$ as an $\mathbb{F}_qL$-module.

We now show that $K_{bc}$ is an irreducible $\mathbb{F}_qL$-module. Let $W$ be a nontrivial $L$-invariant subgroup of $K_{bc}$ and choose a nontrivial element $k(0,b,c) \in W$. Since $(b,c) \ne (0,0)$, there exists some $\alpha,\gamma \in \mathbb{F}_q$ such that $b\alpha+c\gamma=1$: for if $b \ne 0$, we can take $\alpha=b^{-1}$ and $\gamma=0$; whereas if $b=0$, then $c \ne 0$ and we can take $\alpha=0$ and $\gamma=c^{-1}$. Now let $k(0,b',c') \in K_{bc}$ be arbitrary with $(b',c') \ne 0$ and similarly, there exists some $\gamma',\delta' \in \mathbb{F}_q$ such that $b' \delta' + c' \gamma' = 1$. Then we can choose elements $\ell_1,\ell_2 \in L$ where $$\widehat{\ell_1}=\begin{pmatrix} \alpha & c \\ \gamma & b \end{pmatrix} \ \text{ and } \ \widehat{\ell_2}=\begin{pmatrix} b' & c' \\ \gamma' & \delta' \end{pmatrix}.$$ It follows that $k(0,b,c)^{\ell_1\ell_2}=k(0,1,0)^{\ell_2}=k(0,b',c')$, so we conclude that $W=K_{bc}$.

Finally, suppose that $H \le K_{bc}L$ satisfies $\pi(H)=L$. We now claim that $H \cap K_{bc}$ is an $L$-invariant subgroup of $K_{bc}$. Let $x \in H \cap K_{bc}$ and $\ell \in L$ be arbitrary. Then there exists some $k \in K_{bc}$ such that $k\ell \in H$. Since $K_{bc}$ is abelian and $x \in H$, we have $x^\ell=(x^k)^\ell=x^{k\ell} \in H$. Moreover, since $x \in K_{bc} \trianglelefteqslant K_{bc}L$, we also have $x^\ell \in K_{bc}$ and hence $x^\ell \in H \cap K_{bc}$, as required. It follows from the irreducibility of $K_{bc}$ as an $\mathbb{F}_qL$-module that either $H \cap K_{bc}=1 \text{ or } K_{bc}$.
\end{proof}

\section{Proof of Main Results} \label{proof}

Let $\omega$ be a generator of the multiplicative group of $\mathbb{F}_{q^2}$ cyclic of order $q^2-1$, and set $\zeta=\omega^{q-1}$ so that $\zeta$ has order $q+1$ in $\mathbb{F}_{q^2}$, and further set $\lambda=\zeta+\zeta^{-1}=\zeta+\zeta^q$. Note that $\lambda \in \mathbb{F}_q$, since it is the trace of $\zeta$ from $\mathbb{F}_{q^2}$ to $\mathbb{F}_q$. We also claim that $\lambda \ne 1$: if not, then $\zeta+\zeta^{-1}=1$, or equivalently $\zeta^2+\zeta+1=0$. Thus, $\zeta$ must be a primitive cube root of unity, that is, $\zeta$ has order $3$. However, $\zeta$ has order $q+1$, which implies that $q=2$, contradicting our choice of $q>2$. 

We now define the following elements in $\mathrm{GL}_7(q)$. \\[0.1in]
\begin{center}
$t_1:=\begin{pmatrix} 1 & 0 & 0 & 0 & 0 & 0 & 0 \\ 0 & 1 & \lambda & 0 & 0 & 0 & 0 \\ 0 & 0 & 1 & 0 & 0 & 0 & 0 \\ 0 & 0 & 0 & 1 & 0 & 0 & 0 \\ 0 & 0 & 0 & 0 & 1 & \lambda & 0 \\ 0 & 0 & 0 & 0 & 0 & 1 & 0 \\ 0 & 0 & 0 & 0 & 0 & 0 & 1 \end{pmatrix} \quad  t_2:=\begin{pmatrix} 1 & 0 & 0 & 0 & 0 & 0 & 0 \\ 0 & 0 & 1 & 0 & 0 & 0 & 0 \\ 0 & 1 & 0 & 0 & 0 & 0 & 0 \\ 0 & 0 & 0 & 1 & 0 & 0 & 0 \\ 0 & 0 & 0 & 0 & 0 & 1 & 0 \\ 0 & 0 & 0 & 0 & 1 & 0 & 0 \\ 0 & 0 & 0 & 0 & 0 & 0 & 1 \end{pmatrix}$ \\[0.2in]
$t_3:=\begin{pmatrix} 1 & 0 & 0 & 0 & 0 & 1 & 0 \\ 0 & 1 & 1 & 0 & 0 & 0 & 1 \\ 0 & 0 & 1 & 0 & 0 & 0 & 0 \\ 0 & 0 & 0 & 1 & 0 & 0 & 0 \\ 0 & 0 & 0 & 0 & 1 & 1 & 0 \\ 0 & 0 & 0 & 0 & 0 & 1 & 0 \\ 0 & 0 & 0 & 0 & 0 & 0 & 1 \end{pmatrix}$ \\[0.2in] 
$t_4:=\begin{pmatrix} 0 & 0 & 0 & 0 & 0 & 0 & 1 \\ 0 & 0 & 0 & 0 & 1 & 0 & 0 \\ 0 & 0 & 0 & 0 & 0 & 1 & 0 \\ 0 & 0 & 0 & 1 & 0 & 0 & 0 \\ 0 & 1 & 0 & 0 & 0 & 0 & 0 \\ 0 & 0 & 1 & 0 & 0 & 0 & 0 \\ 1 & 0 & 0 & 0 & 0 & 0 & 0 \end{pmatrix} \quad t_5:=\begin{pmatrix} 1 & 0 & 0 & 1 & 0 & 0 & 1 \\ 0 & 1 & 0 & 0 & 1 & 0 & 0 \\ 0 & 0 & 1 & 0 & 0 & 1 & 0 \\ 0 & 0 & 0 & 1 & 0 & 0 & 0 \\ 0 & 0 & 0 & 0 & 1 & 0 & 0 \\ 0 & 0 & 0 & 0 & 0 & 1 & 0 \\ 0 & 0 & 0 & 0 & 0 & 0 & 1 \end{pmatrix}$ \\[0.2in]
\end{center}

We prove Theorem \ref{main} by showing that $\{t_1,t_2,t_3,t_4,t_5\}$ is a C-string of $G$. The following lemmas establish the required properties.

\begin{lem}
$t_i \in G$ for every $1 \le i \le 5$.
\end{lem}
\begin{proof}
It follows from Lemma \ref{mod} that $t_5=t \in G$, $t_1,t_2,t_3 \in C_G(t) \le G$, so it suffices to verify that $t_4 \in G$ using Lemma \ref{form}. Let $\sigma=(1,7)(2,5)(3,6)$. Then $t_4$ permutes the basis vectors according to $\sigma$. 
An inspection of the formula for $f$ shows that $f(u^{t_4} \cdot v^{t_4}, w^{t_4}) = f(u \cdot v , w)$ and hence $t_4 \in G$, as required.
\end{proof}

\begin{lem} \label{ord}
$|t_i|=2$ for every $1 \le i \le 5$ and $|t_it_j|=2$ for every $1 \le i,j \le 5$, where $|i-j| \ge 2$. Furthermore, $|t_1t_2|=q+1$, $|t_2t_3|=3$, $|t_3t_4|=8$, $|t_4t_5|=3$.
\end{lem}

\begin{proof}
Observe that $t_1t_2$ is block diagonal consisting of three $1 \times 1$ identity matrices and two copies of $$\widehat{t_1t_2} = \begin{pmatrix} \lambda & 1 \\ 1 & 0 \end{pmatrix},$$ whose characteristic polynomial is $X^2+\lambda X+1=(X+\zeta)(X+\zeta^{-1})$ over $\mathbb{F}_{q^2}[X]$. It follows that the eigenvalues of $t_1 t_2$ are $1,\zeta,\zeta^{-1}$ and hence $|t_1t_2|$ is equal to the order of $\zeta$, which is $q+1$. The remaining orders follow from direct calculation.
\end{proof}

So far, we have established that $\{t_1,t_2,t_3,t_4,t_5\}$ is a sggi for a subgroup of $G$. Now, observe that $\{t_2,t_3,t_4,t_5\}$ is defined over $\mathbb{F}_2$, and hence is independent of our choice of $q$ and $\lambda$. In fact, it generates the standard copy of $G_2(2) \cong U_3(3):2$. Let $s=(t_3t_4)^2$ and so $s^4=1$. We have

\begin{equation*}
s = \begin{pmatrix} 1 & 0 & 0 & 0 & 0 & 1 & 0 \\ 0 & 1 & 1 & 0 & 0 & 0 & 1 \\ 0 & 0 & 1 & 0 & 0 & 0 & 0 \\ 0 & 0 & 0 & 1 & 0 & 0 & 0 \\ 1 & 0 & 0 & 0 & 1 & 0 & 0 \\ 0 & 0 & 0 & 0 & 0 & 1 & 0 \\ 0 & 0 & 1 & 0 & 0 & 0 & 1 \end{pmatrix} \ \text{ and } \ s^2=\begin{pmatrix} 1 & 0 & 0 & 0 & 0 & 0 & 0 \\ 0 & 1 & 1 & 0 & 0 & 0 & 0 \\ 0 & 0 & 1 & 0 & 0 & 0 & 0 \\ 0 & 0 & 0 & 1 & 0 & 0 & 0 \\ 0 & 0 & 0 & 0 & 1 & 1 & 0 \\ 0 & 0 & 0 & 0 & 0 & 1 & 0 \\ 0 & 0 & 0 & 0 & 0 & 0 & 1 \end{pmatrix}. \\[0.1in]
\end{equation*}

Next, we record some properties of the subgroup $\langle t_2,t_3,t_4,t_5 \rangle$.

\begin{lem} \label{g2345}
$\langle t_2,t_3,t_4,t_5 \rangle \cong G_2(2)$ and $\{t_2,t_3,t_4,t_5\}$ is a C-string. Furthermore,
\begin{enumerate}
\item $C_{\langle t_2,t_3,t_4 \rangle}(t_5) = \langle t_2,t_3,s^2 \rangle \cong \mathrm{Sym}(4)$,
\item $\langle t_2,t_3 \rangle \cong \mathrm{Sym}(3)$ is maximal in $\langle t_2,t_3,s^2 \rangle$, 
\item $\langle t_2,t_3,t_4 \rangle \cong \mathrm{PSL}_2(7):2$ is maximal in $\langle t_2,t_3,t_4,t_5 \rangle$.
\end{enumerate}
\end{lem}

\begin{proof}
The lemma follows from direct calculation in \textsc{Magma} \cite{magma}. The assertions on maximality can also be determined using known maximal subgroup structure of symmetric groups and $G_2(2)$ as recorded in the ATLAS of Finite Groups \cite{atlas}.
\end{proof}

We now focus on establishing that $\{t_1,t_2,t_3,t_4\}$ is a C-string.

\begin{lem} \label{g123s}
$s^2 \notin \langle t_1,t_2,t_3 \rangle$.
\end{lem}

\begin{proof}
Let $v=\sum_{i=1}^7 v_i y_i \in V$ and consider the quadratic form given by
$$Q(v)=v_2^2+v_3^2+\lambda v_2v_3 + (1+\lambda)v_7^2.$$ Observe that 
$Q(v^{t_1})=Q(v^{t_2})=Q(v^{t_3})=Q(v)$ 
and hence $\langle t_1,t_2,t_3 \rangle$ preserves $Q$. On the other hand, since $\lambda \ne 1$,
\begin{align*}
Q(v^{s^2}) &= v_2^2+(v_2+v_3)^2+\lambda v_2(v_2+v_3)+(1+\lambda)v_7^2 \\ &= Q(v)+(1+\lambda) v_2^2.
\end{align*}
Thus, $Q(v^{s^2}) \ne Q(v)$ whenever $v_2 \ne 0$. It follows that $s^2$ does not preserve $Q$ and hence, $s^2 \notin \langle t_1,t_2,t_3 \rangle$.
\end{proof}

\begin{lem} \label{g123p}
$\langle t_1,t_2,s^2 \rangle = L$ and $\langle t_1,t_2,t_3,s^2 \rangle = K_{bc}L \cong q^2:\mathrm{SL}_2(q)$.
\end{lem}

\begin{proof}
Note that $\langle t_1,t_2,s^2 \rangle \le L$ and $\langle t_1,t_2\rangle \cong D_{q+1}$, which is a maximal subgroup of $L \cong \mathrm{SL}_2(q)$ by \cite[Table 8.1]{BrayHoltRoneyDougal-maxbible}. Observing the $2 \times 2$ block, we see that $$\widehat{(t_1t_2)^{s^2}}=\begin{pmatrix} 1+\lambda & \lambda \\ 1 & 1 \end{pmatrix} \ \text{ and } \ \widehat{(t_1t_2)^{-1}} = \begin{pmatrix} 0 & 1 \\ 1 & \lambda \end{pmatrix}.$$ Since $\lambda \ne 1$, it follows that $(t_1t_2)^{s^2} \ne (t_1t_2)^{-1}$ and hence $s^2 \notin \langle t_1,t_2 \rangle$. We conclude that $\langle t_1,t_2,s^2 \rangle = L$ by maximality. 

Let $H=\langle t_1,t_2,t_3,s^2 \rangle$. Now, $L \le H \le K_{bc}L$ and $H=(H \cap K_{bc})L$. However, we have $1 \ne t_3s^2=k(0,1,1) \in H \cap K_{bc}$. This rules out the possibility that $H \cap K_{bc}=1$ and so $H \cap K_{bc}=K_{bc}$ by Lemma \ref{mod}. Therefore, $H = K_{bc}L \cong q^2:\mathrm{SL}_2(q)$.
\end{proof}

\begin{lem} \label{g123}
$\langle t_1,t_2,t_3 \rangle \cong \mathrm{SL}_2(q)$ and $\{t_1,t_2,t_3\}$ is a C-string. 
\end{lem}

\begin{proof}
Let $H=\langle t_1,t_2,t_3 \rangle$. Consider the natural projection map $\pi : K_{bc}L \to L$. Observe that $t_1,t_2 \in L$ and $\pi(t_3)=s^2$, so $$\pi(H) = \langle \pi(t_1),\pi(t_2),\pi(t_3) \rangle = \langle t_1,t_2,s^2 \rangle = L.$$ Thus, $\pi\vert_H$, the restriction of $\pi$ to $H$ is surjective. We now show that $\pi\vert_H$ is injective by calculating its kernel. Note that $\ker \pi\vert_H = H \cap K_{bc}$. Contrary to Lemma \ref{g123p}, here $1 \ne t_3s^2=k(0,1,1) \notin H$, otherwise $s^2 \in H$, contradicting Lemma \ref{g123s}. This rules out the possibility that $H \cap K_{bc}=K_{bc}$ and so $H \cap K_{bc}=1$ by Lemma \ref{mod}. It follows that $\ker\pi\vert_H=1$ and $\pi\vert_H$ is an isomorphism. Therefore, $H \cong L \cong \mathrm{SL}_2(q)$. 

Note that $H$ is a smooth quotient of the $[q+1,3]$ Coxeter
group by Lemma \ref{ord}. Since $q > 2$ is even, $H \cong \mathrm{SL}(2,q) \cong \mathrm{PSL}(2,q)$ is a non-abelian simple group and $q+1 > 3$. Thus,
\cite[Corollary 4.2]{ConderOliveros} implies that $\{t_1,t_2,t_3\}$ is a
C-string.
\end{proof}

\begin{lem} \label{g1234}
$\{t_1,t_2,t_3,t_4\}$ is a C-string.
\end{lem}

\begin{proof}
Let $H=\langle t_1,t_2,t_3 \rangle \cap \langle t_2,t_3,t_4 \rangle$. It is immediate that $\langle t_2,t_3 \rangle \le H$. Since $\langle t_1,t_2,t_3 \rangle \le C_G(t_5)$, we have $H \le C_G(t_5) \cap \langle t_2,t_3,t_4 \rangle = C_{\langle t_2,t_3,t_4 \rangle} (t_5) = \langle t_2,t_3,s^2 \rangle$ by Lemma \ref{g2345}(1). Moreover, $\langle t_2,t_3 \rangle$ is maximal in $\langle t_2,t_3,s^2 \rangle$ by Lemma \ref{g2345}(2). However, $s^2 \notin \langle t_1,t_2,t_3 \rangle$ by Lemma \ref{g123s} and hence $s^2 \notin H$. We conclude that $H=\langle t_2,t_3 \rangle$ by maximality. Since $\{t_1,t_2,t_3\}$ and $\{t_2,t_3,t_4\}$ are C-strings by Lemmas \ref{g2345} and \ref{g123}, it follows that $\{t_1,t_2,t_3,t_4 \}$ is a C-string by Lemma \ref{ip}(1).
\end{proof}

We now establish that $\{t_1,t_2,t_3,t_4,t_5\}$ is a C-string. Then, to establish that it generates $G$, we show that $\langle t_1,t_2,t_3,t_4\rangle$ is a maximal subgroup of $G$.

\begin{lem} \label{t5}
$t_5 \notin \langle t_1,t_2,t_3,t_4 \rangle$.
\end{lem}

\begin{proof}
    Consider the subspace of $V$ given by $$U=\mathrm{span}_{\mathbb{F}_q}(y_1,y_2,y_3,y_5,y_6,y_7).$$ 
    Observe that $U^{t_1},U^{t_2},U^{t_3},U^{t_4} \subseteq U$
and hence $\langle t_1,t_2,t_3,t_4 \rangle$ preserves $U$. On the other hand, $y_1^{t_5}=y_1+y_4+y_7\notin U$. It follows that $t_5$ does not preserve $U$ and hence $t_5 \notin \langle t_1,t_2,t_3,t_4 \rangle$. 
\end{proof}

\begin{lem} \label{g12345}
$\{t_1,t_2,t_3,t_4,t_5\}$ is a C-string.
\end{lem}

\begin{proof}
Since $\{t_1,t_2,t_3,t_4\}$ and $\{t_2,t_3,t_4,t_5\}$ are C-strings by Lemmas \ref{g2345} and \ref{g1234}, as well as $t_5 \notin \langle t_1,t_2,t_3,t_4 \rangle$ and $\langle t_2,t_3,t_4 \rangle$ maximal in $\langle t_2,t_3,t_4,t_5 \rangle$ by Lemmas \ref{g2345}(3) and \ref{t5}, it follows that $\{t_1,t_2,t_3,t_4,t_5\}$ is a C-string by Lemma \ref{ip}(2).
\end{proof}

\begin{lem} \label{g1234s}
$\langle t_1,t_2,t_3,s \rangle \cong \mathrm{SL}_3(q)$.
\end{lem}

\begin{proof}
Let $H=\langle t_1,t_2,t_3,s \rangle$. Consider the subspaces of $V$ given by $$U=\mathrm{span}_{\mathbb{F}_q}(y_1,y_5,y_6) \ \ \text{ and } \ \ \widehat{U}=\mathrm{span}_{\mathbb{F}_q}(y_7,y_3,y_2).$$ Observe that $U^{t_1},U^{t_2},U^{t_3},U^{s} \subseteq U$ and similarly $\widehat{U}^{t_1},\widehat{U}^{t_2},\widehat{U}^{t_3},\widehat{U}^{s} \subseteq \widehat{U}$ and hence $H$ preserves $U$ and $\widehat{U}$. We also observe that $y_4$ is fixed by $H$.

Consider the induced representations $\rho : H \to \mathrm{GL}(U)$ and $\widehat{\rho} : H \to \mathrm{GL}(\widehat{U})$. With respect to the ordered basis $\mathcal{B}=(y_1,y_5,y_6,y_4,y_7,y_3,y_2)$, each element $h \in H$ has the form $[h]_{\mathcal{B}}=\rho(h) \oplus I_1 \oplus \widehat{\rho}(h) \in \mathrm{GL}_7(q)$, and we have $\widehat{\rho}(h)=\rho(h)^{-T}$ by direct calculation on the generators of $H$. Let $g \in \mathrm{ker}(\rho)$. Then $[g]_{\mathcal{B}}=I_3 \oplus I_1 \oplus {I_3}^{-T} = I_7$ and hence $g=1_H$, that is, $\ker(\rho)=1_H$. It follows that $\rho$ is a faithful representation and $H \cong \rho(H)$. Since the generators of $\rho(H)$ each have determinant 1, we deduce that $\rho(H) \le \mathrm{SL}_3(q)$. Suppose to the contrary that $\rho(H) < \mathrm{SL}_3(q)$.

Let $H_0=\langle t_1,t_2,t_3,s^2 \rangle \le H$. Then, $\rho(H_0) \le \rho(H) \le M$ for some maximal subgroup $M$ of $\mathrm{SL}_3(q)$, and $\rho(H_0) \cong H_0 \cong q^2 : \mathrm{SL}_2(q)$ by the faithfulness of $\rho$ and Lemma \ref{g123p}. Again, by direct calculation on the generators of $H_0$, we see that $H_0$ stabilizes the two-dimensional space $U_0:=\mathrm{span}_{\mathbb{F}_q}(y_5,y_6)$. It follows that $\rho(H_0) \le P:=\mathrm{Stab}_{\mathrm{SL}_3(q)}(U_0)$,
where $P$ is a maximal parabolic subgroup of $\mathrm{SL}_3(q)$ of shape $q^2:\mathrm{GL}_2(q)$. Since $\rho(H_0)$ has index $q-1>1$ in $P$, we have $\rho(H_0) < P$, and hence $\rho(H_0)$ is not a maximal subgroup of $\mathrm{SL}_3(q)$.


We claim that $M=P$; that is, $P$ is the unique maximal overgroup of $\rho(H_0)$ in $\mathrm{SL}_3(q)$. Observe that $$[\mathrm{SL}_3(q):M] < [\mathrm{SL}_3(q):\rho(H_0)]=\frac{|\mathrm{SL}_3(q)|}{|q^2:\mathrm{SL}_2(q)|}=\frac{q^3(q^3-1)(q^2-1)}{q^2(q(q^2-1))} = q^3-1.$$
When $q>4$, inspecting the maximal subgroups of $\mathrm{SL}_3(q)$ given in \cite[Tables 8.3 \& 8.4]{BrayHoltRoneyDougal-maxbible} and calculating their indices in $\mathrm{SL}_3(q)$ shows that the only possibility for the shape of $M$ is the parabolic subgroup $q^2:\mathrm{GL}_2(q)$. When $q=4$, there is the additional possibility of an exceptional nonparabolic subgroup $3.A_6$ of order 1080, which cannot contain $\rho(H_0)\cong4^2:\mathrm{SL}_2(4)$ of order 960. Thus, for every even $q>2$, $M \cong q^2:\mathrm{GL}_2(q)$ and stabilizes either a one- or two-dimensional subspace of $U$.

To prove our claim, it remains to show that $\rho(H_0)$ does not stabilize a one-dimensional subspace of $U$, and $U_0$ is the unique $\rho(H_0)$-invariant two-dimensional subspace of $U$. Suppose that $W$ is a $\rho(H_0)$-invariant one-dimensional subspace of $U$. Then $W=\mathrm{span}_{\mathbb{F}_q}(w)$ where $w=w_1y_1+w_5y_5+w_6y_6$ for some $(w_1,w_5,w_6) \ne (0,0,0)$. Observe that
\begin{align*}
w^{\rho(t_2)}&=w_1y_1+w_6y_5+w_5y_6, \\
w^{\rho(t_3)}&=w_1y_1+w_5y_5+(w_1+w_5+w_6)y_6, \\
w^{\rho(s^2)}&=w_1y_1+w_5y_5+(w_5+w_6)y_6.
\end{align*}
First, suppose that $w_1 \ne 0$. Since $\langle w \rangle$ is $\rho(H_0)$-invariant, we have $w^{\rho(h)}=\lambda_h w$ for every $h \in \{t_2,t_3,s^2\}$; comparing the coefficient of $y_1$ gives $\lambda_h=1$ in each case. From $t_2$, we obtain $w_5=w_6$; from $t_3$, we have $w_1+w_5+w_6=w_6$, and hence $w_1=w_5$; finally $s^2$ gives $w_6=w_5+w_6$, so $w_5=0$. It follows that $w_1=w_5=w_6=0$, a contradiction. Now, suppose that $w_1=0$. Again, $w^{\rho(s^2)}=\lambda w$ gives $w_5=\lambda w_5$ and $w_5+w_6=\lambda w_6$. If $w_5 \ne 0$, then $\lambda=1$ by the first equation, while the second equation implies $w_5=0$, another contradiction. It follows that $w_5=0$, so $w \in \langle y_6 \rangle$, which is not invariant under $\rho(t_2)$. Thus, $\rho(H_0)$ stabilizes no one-dimensional subspace of $U$. Finally, suppose that $U_1 \ne U_0$ is a $\rho(H_0)$-invariant two-dimensional subspace of $U$. Since $\dim U = 3$, we have $\dim(U_0 \cap U_1)=1$ and $U_0 \cap U_1$ is a $\rho(H_0)$-invariant one-dimensional subspace of $U$, contradicting the preceding result. Therefore, $U_0$ is the unique $\rho(H_0)$-invariant two-dimensional subspace of $U$. It follows that $M=P=\mathrm{Stab}_{\mathrm{SL}_3(q)}(U_0)$, as claimed.

Hence, $\rho(H) \le P$. However, $s \in H$ and $y_5 \in U_0$, but $y_5^{\rho(s)}=y_1+y_5 \notin U_0$. Thus, $\rho(s) \notin P$, a contradiction. Therefore, $\rho(H)=\mathrm{SL}_3(q)$. Since $\rho$ is faithful, we conclude that $H \cong \mathrm{SL}_3(q)$. We also note that $\langle t_1,t_2,t_3,t_4 \rangle = \langle H,t_4 \rangle$ 
and $U^{t_4} = \widehat{U}$; here $t_4$ swaps the two $H$-invariant subspaces $U$ and $\widehat{U}$, and hence $t_4 \not\in H$. Moreover, we have $t_1^{t_4}=t_1$, $t_2^{t_4}=t_2$, $t_3^{t_4}=t_3s$, $s^{t_4}=s^{-1}$, so $t_4$ normalizes $H$. Furthermore, $$t_2t_4=\begin{pmatrix} 0 & 0 & 0 & 0 & 0 & 0 & 1 \\ 0 & 0 & 0 & 0 & 0 & 1 & 0 \\ 0 & 0 & 0 & 0 & 1 & 0 & 0 \\ 0 & 0 & 0 & 1 & 0 & 0 & 0 \\ 0 & 0 & 1 & 0 & 0 & 0 & 0 \\ 0 & 1 & 0 & 0 & 0 & 0 & 0 \\ 1 & 0 & 0 & 0 & 0 & 0 & 0 \end{pmatrix}$$ induces the graph automorphism of $\mathrm{SL}_3(q)$. Now, conjugation by $t_2 \in H$ is an inner automorphism, so $t_4$ induces a non-trivial outer automorphism in the same class as the graph automorphism. We conclude that $\langle H,t_4 \rangle=H \rtimes \langle t_4 \rangle \cong \mathrm{SL}_3(q):2$.
\end{proof}

We are now ready to prove our main theorem.

\begin{proof}[Proof of Theorem \ref{main}]
It follows from Lemma \ref{g1234s} that $\langle t_1,t_2,t_3,t_4 \rangle \cong \mathrm{SL}_3(q):2$, which is a maximal subgroup of $G$ by \cite[Table 8.30]{BrayHoltRoneyDougal-maxbible} (see also \cite[Theorem 2.3]{Cooperstein-g2max}). However, $t_5 \notin \langle t_1,t_2,t_3,t_4 \rangle$ by Lemma \ref{t5}. We conclude that $\langle t_1,t_2,t_3,t_4,t_5 \rangle = G$ by maximality. Thus, $\{t_1,t_2,t_3,t_4,t_5\}$ is a rank five C-string of $G$ with Schl\"{a}fli type $[q+1,3,8,3]$ by Lemmas \ref{ord} and \ref{g12345}. Finally, we have $\langle t_1,t_2,t_3 \rangle \cong \mathrm{SL}_2(q)$ and $\langle t_2,t_3,t_4,t_5 \rangle \cong G_2(2)$ by Lemmas \ref{g2345} and \ref{g123}.
\end{proof}

\begin{rem}
It can be shown that after conjugation by a suitable change-of-basis matrix, taking the dual and setting the parameter $t=\lambda$, the rank five C-string of Gvozdev and Nuzhin \cite{GvozdevNuzhin-g2} corresponds to the C-string $\{t_1,t_2,t_3s^2,t_4,t_5\}$. Thus, the two constructions differ only in the choice of the third generator.
\end{rem}

Finally, we recall a variant of the rank reduction theorem in \cite{BL-rankreduction}, which we will use to obtain string C-group representations of $G$ of ranks three and four.

\begin{lem}{\cite[Corollary 1.2]{BL-rankreduction}} \label{rankred}
Let $(\Gamma,\{ s_1,\dots,s_r \})$ be a non-degenerate string C-group of rank $r \ge 4$. If $s_3s_4$ has odd order, then $(\Gamma,\{ s_2, s_1 s_3, s_4, \dots,  s_r \})$ is a string C-group of rank $r-1$.
\end{lem}

\begin{proof}[Proof of Corollary \ref{notsomain}]
It follows from Theorem \ref{main} that $G$ admits a rank five C-string and we now apply Lemma \ref{rankred} to its dual $\{t_5,t_4,t_3,t_2,t_1\}$. Taking $s_i=t_{6-i}$ for every $1 \le i \le 5$, we observe that $s_3s_4=t_3t_2$ has odd order three by Lemma \ref{ord} and hence, $\{t_4,t_3t_5,t_2,t_1\}$ is a rank four C-string of $G$. Applying Lemmas \ref{ord} and \ref{rankred} again, since $t_2t_1$ has odd order $q+1$, it follows that $\{t_3t_5,t_2t_4,t_1\}$ is a rank three C-string of $G$. A direct calculation gives $|t_4t_3t_5|=12$, $|t_3t_5t_2|=6$ and $|t_3t_5t_2t_4|=8$. Finally, since $t_4$ commutes with $t_1$ and $t_2$, and $t_2t_1$ has odd order $q+1$, we also have $|(t_2t_4)t_1|=2(q+1)$, which gives the stated Schl\"{a}fli types.
\end{proof}

\begin{rem}
Gvozdev and Nuzhin \cite{GvozdevNuzhin-g2} constructed the ranks three and four C-strings of $G$ without appealing to the rank reduction theorem.
\end{rem}

\section{Computational Results and Concluding Remarks} \label{conclusion}

We performed an exhaustive computational search for string C-group representations of $G_2(q)$ of ranks four and above for every prime power $3 \le q \leqslant 11$; a census is given in Table \ref{data}. Moreover, for $3 \le q \le 9$, a detailed description of the rank five string C-group representations of $G_2(q)$ when $q$ is even and the rank four string C-group representations of $G_2(q)$ when $q$ is odd are given in the Appendix as Tables \ref{g24}--\ref{g29}; we omit describing $q=11$ owing to the large number of examples, but nevertheless the examples are available from the author upon request.


\begin{table}[h]
\centering
\caption{Number of string C-group representations of $G_2(q)$ of ranks four and above (up to isomorphism and duality) for every prime power $3 \le q \leqslant 11$.}
\begin{tabular}{c|c|c|c} \label{data}
$q$ & Rank 4 & Rank 5 & Rank $\geqslant 6$ \\
\hline
3  & 1 & 0 & 0 \\
4  & 67 & 14 & 0 \\
5  & 16 & 0 & 0 \\
7  & 21 & 0 & 0 \\
8  & 308 & 4 & 0 \\
9  & 23 & 0 & 0 \\
11 & 105 & 0 & 0 \\
\end{tabular}
\end{table}

The uniform construction given in Theorem \ref{main} appears to be somewhat exceptional among the rank five C-strings of $G_2(q)$. Indeed, 12 of the 14 rank five C-strings of $G_2(4)$ arise from rank four C-strings of the sporadic simple group $J_2$, an exceptional maximal subgroup of $G_2(4)$. The remaining two examples for $q=4$ and all four examples for $q=8$ belong to two uniform families. In both cases, $\langle t_1,t_2,t_3,t_4 \rangle \cong \mathrm{SL}_3(q):2$ and $\langle t_2,t_3,t_4,t_5 \rangle \cong G_2(2)$, while $\langle t_1,t_2,t_3 \rangle \cong q^2:D_{q\pm1}$ or $\mathrm{SL}_2(q)$; the former case corresponds to the construction given by Gvozdev and Nuzhin \cite{GvozdevNuzhin-g2}, whereas latter case corresponds to the construction given in this paper. This suggests that five might be the maximal rank of a string C-group representation of $G_2(q)$ in even characteristic.

On the other hand, we obtained only rank four C-strings for odd prime powers $q\le 11$, and these examples exhibit considerably less uniform behavior. Nevertheless, it is straightforward to obtain an upper bound on the rank. Recall that in any rank $r$ C-string, the $\lceil{\frac{r}{2}}\rceil$ alternating generators commute by the string property and are independent by the intersection property. Since $G_2(q)$ has 2-rank 3 when $q$ is odd (see, for example, \cite[Theorem 4.10.5]{GorensteinLyonsSolomon-cfsgbible-vol3}), it follows that $r \le 6$.

We close with two natural questions.

\begin{qn}
When $q$ is even, are three, four, and five the only possible ranks of C-strings of $G_2(q)$? That is, does $G_2(q)$ admit no C-strings of rank six or above?
\end{qn}

\begin{qn}
When $q$ is odd, what is the general behavior of C-strings of $G_2(q)$? In particular, does $G_2(q)$ admit a rank five or six C-string?
\end{qn}

\section*{Acknowledgments}

The author thanks Peter Rowley for suggesting this research problem and James Bryden for several helpful discussions.

\section*{Appendix: Descriptions of String C-Group Representations of $G_2(q)$}

Throughout, $G_{i_1,\dots,i_k}$ denotes the parabolic subgroup $\langle t_{i_1},\dots,t_{i_k}\rangle$. The tables record the isomorphism types of these subgroups, as returned by the \texttt{GroupName} command in \textsc{Magma} \cite{magma}. The examples are ordered lexographically by Schl\"{a}fli type, with self-dual string C-group representations grouped together at the end; for such cases $G_{123} \cong G_{234}$ and so we record the corresponding entries in a shared column. We also group together non-isomorphic string C-group representations with identical Schl\"{a}fli types and parabolic subgroups for compactness. Finally, we note that $\mathrm{SO}_3(7) \cong \mathrm{PSL}_2(7):2 \cong \mathrm{PGL}_2(7)$.

For $q=3$, there is a unique rank four string C-group representation of $G_2(q)$ up to isomorphism and duality. It has Schl\"{a}fli type $[3,8,3]$ and $G_{123}\sim \mathrm{PSL}_3(3).2$ and $G_{234} \sim \mathrm{SO}_3(7)$. The remaining tables are listed in increasing order of $q$.

\begin{table}
\caption{Complete list of the 14 rank five string C-group representations of $G_2(4)$ up to isomorphism and duality.} \label{g24}
\begin{tabular}{C|C|C|C|C|C|C}
 & \text{Schl\"{a}fli type} & G_{1234} & G_{2345} & G_{123} & G_{234} & G_{345} \\ \hline
1 & [3,3,8,3] & \mathrm{SL}_3(4).2 & J_2 & S_4 & 3.A_6.2 & \mathrm{SL}_2(3).2.2 \\ 
2,3 & [3,3,8,5] & \mathrm{SL}_3(4).2 & J_2 & S_4 & 3.A_6.2 & 2.2^4.D_5 \\ 
4 & [3,5,8,3] & \mathrm{SL}_3(4).2 & J_2 & \mathrm{SL}_2(4) & 3.A_6.2 & \mathrm{SL}_2(3).2.2 \\ 
5,6 & [3,5,8,5] & \mathrm{SL}_3(4).2 & J_2 & \mathrm{SL}_2(4) & 3.A_6.2 & 2.2^4.D_5 \\ 
7 & [5,3,8,3] & \mathrm{SL}_3(4).2 & G_2(2) & \mathrm{SL}_2(4) & \mathrm{SO}_3(7) & \mathrm{SL}_2(3).2.2 \\ 
8 & [5,3,8,3] & \mathrm{SL}_3(4).2 & J_2 & \mathrm{SL}_2(4) & 3.A_6.2 & \mathrm{SL}_2(3).2.2 \\ 
9,10 & [5,3,8,5] & \mathrm{SL}_3(4).2 & J_2 & \mathrm{SL}_2(4) & 3.A_6.2 & 2.2^4.D_5 \\ 
11 & [5,4,8,5] & \mathrm{SL}_3(4).2 & G_2(2) & 2^4.D_5 & \mathrm{SO}_3(7) & \mathrm{SL}_2(3).2.2 \\
12 & [5,5,8,3] & \mathrm{SL}_3(4).2 & J_2 & 2^4.D_5 & 3.A_6.2 & \mathrm{SL}_2(3).2.2 \\ 
13,14 & [5,5,8,5] & \mathrm{SL}_3(4).2 & J_2 & 2^4.D_5 & 3.A_6.2 & 2.2^4.D_5 \\ 
\end{tabular}
\end{table}

\begin{table}
\caption{Complete list of the 16 rank four string C-group representations of $G_2(5)$ up to isomorphism and duality.} \label{g25}
\begin{tabular}{C|C|C|C}
 & \text{Schl\"{a}fli type} & G_{123} & G_{234} \\ \hline
1 & [3,8,5] & \mathrm{SL}_3(5).2 & 3.A_6.2 \\  
2 & [3,8,5] & 3.\mathrm{SO}_3(7) & \mathrm{SL}_3(5).2 \\  
3 & [3,8,6] & 3.\mathrm{SO}_3(7) & \mathrm{SL}_3(5).2 \\  
4 &[3,10,4] & 3 \times A_5 & 5^2.\mathrm{He}_5.D_4 \\  
5 & [3,10,4] & 3.A_6.2^2 & 5^2.\mathrm{He}_5.D_6 \\  
6 & [3,10,5] & 5^2.\mathrm{He}_5.D_6 & 3.A_6.2 \\  
7 & [4,8,5] & \mathrm{SO}_3(7) & \mathrm{SL}_3(5).2 \\  
8 & [4,10,4] & 5^2.\mathrm{He}_5.D_4 & 3.A_7.2 \\  
9 & [4,10,6] & 5^2.\mathrm{He}_5.D_4 & 5^2.\mathrm{He}_5.D_6 \\  
10 & [5,6,6] & 3.A_7.2 & \mathrm{SL}_3(5).2 \\  
11 & [5,10,6] & 3.A_7.2 & 5^2.\mathrm{He}_5.D_6 \\  
12 & [6,6,6] & \mathrm{SL}_3(5).2 & \mathrm{SO}_3(7) \\  
13 & [3,10,3] & \multicolumn{2}{C}{5^2.\mathrm{He}_5.D_6} \\  
14 & [5,15,5] & \multicolumn{2}{C}{\mathrm{SL}_2(5).A_5} \\  
15 & [6,10,6] & \multicolumn{2}{C}{(5 \times \mathrm{He}_5):D_6} \\  
16 & [15,5,15] & \multicolumn{2}{C}{\mathrm{SL}_2(5).A_5} \\  
\end{tabular}
\end{table}

\begin{table}
\caption{Complete list of the 21 rank four string C-group representations of $G_2(7)$ up to isomorphism and duality.} \label{g27}
\begin{tabular}{C|C|C|C}
 & \text{Schl\"{a}fli type} & G_{123} & G_{234} \\ \hline
1 & [3,6,7] & 8^2.D_6 & \mathrm{SL}_3(7).2 \\ 
2 & [3,6,7] & S_3 \times S_4 & \mathrm{SU}_3(7).2 \\ 
3 & [3,8,6] & \mathrm{SO}_3(7) & Q_{16}.\mathrm{SO}_3(7) \\ 
4 & [3,14,4] & 7^2.\mathrm{He}_7.D_6 & 7^2.\mathrm{He}_7.D_4 \\ 
5 & [3,14,8] & S_3 \times \mathrm{SO}_3(7) & 7^2.\mathrm{He}_7.D_8 \\ 
6 & [6,6,8] & \mathrm{SO}_3(7) & \mathrm{SU}_3(7).2 \\ 
7 & [6,14,8] & (7 \times \mathrm{He}_7):D_6 & \mathrm{He}_7.7^2.D_8 \\ 
8 & [7,8,8] & \mathrm{SL}_2(7).\mathrm{SO}_3(7) & Q_{16}.\mathrm{SO}_3(7) \\ 
9 & [7,21,8] & \mathrm{SL}_2(7).\mathrm{SO}_3(7) & \mathrm{SL}_2(7).\mathrm{SO}_3(7) \\ 
10 & [8,3,8] & \mathrm{SO}_3(7) & 2 \times \mathrm{SO}_3(7) \\ 
11 & [8,7,21] & \mathrm{SL}_2(7).\mathrm{SO}_3(7) & \mathrm{SL}_2(7).\mathrm{SO}_3(7) \\ 
12 & [8,8,21] & Q_{16}.\mathrm{SO}_3(7) & \mathrm{SL}_2(7).\mathrm{SO}_3(7) \\ 
13 & [6,8,6] & \multicolumn{2}{C}{Q_{16}.\mathrm{SO}_3(7)}  \\ 
14 & [6,14,6] & \multicolumn{2}{C}{7^2.\mathrm{He}_7.D_6}  \\ 
15 & [7,8,7] & \multicolumn{2}{C}{\mathrm{SL}_2(7).\mathrm{SO}_3(7)}  \\ 
16 & [8,6,8] & \multicolumn{2}{C}{Q_{16}.\mathrm{SO}_3(7)}  \\ 
17 & [8,7,8] & \multicolumn{2}{C}{\mathrm{SL}_2(7).\mathrm{SO}_3(7)}  \\ 
18 & [8,14,8] & \multicolumn{2}{C}{\mathrm{He}_7.7^2.D_8}  \\ 
19 & [8,14,8] & \multicolumn{2}{C}{7^2.\mathrm{He}_7.D_8}  \\ 
20 & [8,21,8] & \multicolumn{2}{C}{\mathrm{SL}_2(7).\mathrm{SO}_3(7)}  \\ 
21 & [21,8,21] & \multicolumn{2}{C}{\mathrm{SL}_2(7).\mathrm{SO}_3(7)}  \\ 
\end{tabular}
\end{table}

\begin{table}
\caption{Complete list of the 4 rank five string C-group representations of $G_2(8)$ up to isomorphism and duality.} \label{g28}
\begin{tabular}{C|C|C|C|C|C|C}
 & \text{Schl\"{a}fli type} & G_{1234} & G_{2345} & G_{123} & G_{234} & G_{345} \\ \hline
1 & [7,3,8,3] & \mathrm{SL}_3(8).2 & G_2(2) & \mathrm{SL}_2(8) & \mathrm{SO}_3(7) & \mathrm{SL}_2(3).2.2 \\  
2 & [7,4,8,3] & \mathrm{SL}_3(8).2 & G_2(2) & 2^6.D_7 & \mathrm{SO}_3(7) & \mathrm{SL}_2(3).2.2 \\  
3 & [9,3,8,3] & \mathrm{SL}_3(8).2 & G_2(2) & \mathrm{SL}_2(8) & \mathrm{SO}_3(7) & \mathrm{SL}_2(3).2.2 \\  
4 & [9,4,8,3] & \mathrm{SL}_3(8).2 & G_2(2) & 2^6.D_9 & \mathrm{SO}_3(7) & \mathrm{SL}_2(3).2.2 \\  
\end{tabular}
\end{table}

\begin{table}
\caption{Complete list of the 23 rank four string C-group representations of $G_2(9)$ up to isomorphism and duality.} \label{g29}
\begin{tabular}{C|C|C|C}
 & \text{Schl\"{a}fli type} & G_{123} & G_{234} \\ \hline
1 & [4,6,5] & 3.S_3 \wr 2 & \mathrm{SU}_3(9).2 \\ 
2 & [4,6,5] & \mathrm{SU}_3(9).2 & 3^2.3^4.D_{10} \\ 
3 & [4,6,8] & 3.S_3 \wr 2 & \mathrm{SU}_3(9).2 \\ 
4 & [4,6,8] & \mathrm{SU}_3(9).2 & 3^2.3^4.D_8 \\ 
5 & [4,6,10] & 3.S_3 \wr 2 & \mathrm{SL}_3(9).2 \\ 
6 & [4,6,10] & \mathrm{SL}_3(9).2 & 3^2.3^4.D_{10} \\ 
7 & [5,6,5] & 3^2.3^4.D_{10} & \mathrm{SL}_3(9).2 \\ 
8 & [5,6,8] & 3^2.3^4.D_{10} & \mathrm{SL}_3(9).2 \\ 
9 & [5,6,8] & 3^2.3^4.D_{10} & \mathrm{SU}_3(9).2 \\ 
10 & [5,6,8] & \mathrm{SL}_3(9).2 & 3^2.3^4.D_8 \\ 
11 & [5,6,8] & \mathrm{SU}_3(9).2 & 3^2.3^4.D_8 \\ 
12 & [5,6,10] & 3^2.3^4.D_{10} & \mathrm{SL}_3(9).2 \\ 
13 & [5,6,10] & 3^2.3^4.D_{10} & \mathrm{SU}_3(9).2 \\ 
14 & [5,6,10] & 3^2.3^4.D_{10} & \mathrm{SU}_3(9).2 \\ 
15 & [5,6,10] & \mathrm{SL}_3(9).2 & 3^2.3^4.D_{10} \\ 
16 & [5,6,10] & \mathrm{SU}_3(9).2 & 3^2.3^4.D_{10} \\ 
17 & [8,6,8] & 3^2.3^4.D_8 & \mathrm{SL}_3(9).2 \\ 
18 & [8,6,8] & 3^2.3^4.D_8 & \mathrm{SU}_3(9).2 \\ 
19 & [8,6,10] & 3^2.3^4.D_8 & 3^2.3^4.D_{10} \\ 
20 & [8,6,10] & \mathrm{SU}_3(9).2 & 3^2.3^4.D_{10} \\ 
21 & [8,6,10] & 3^2.3^4.D_8 & \mathrm{SU}_3(9).2 \\ 
22 & [10,6,10] & 3^2.3^4.D_{10} & \mathrm{SL}_3(9).2 \\ 
23 & [5,6,5] & \multicolumn{2}{C}{3^2.3^4.D_{10}} \\ 
\end{tabular}
\end{table}

We highlight that the rank four string C-group representations of $G_2(9)$ involves only a small number of configurations for its parabolic subgroups. We summarize the frequency with which each unordered pair $\{G_{123},G_{234}\}$ occur in Table \ref{g29a}.

\begin{table}[ht] \centering \caption{Configurations of parabolic subgroups among the 23 rank four string C-group representations of $G_2(9)$ up to isomorphism and duality.} \label{g29a} \begin{tabular}{C|C} \hline 
\{G_{123},G_{234}\} & \text{Number} \\ \hline 
\{3.S_3\wr2,\mathrm{SL}_3(9).2\} & 1 \\ 
\{3.S_3\wr2,\mathrm{SU}_3(9).2\} & 2 \\ 
\{3^2.3^4.D_8,3^2.3^4.D_{10}\} & 1 \\ 
\{3^2.3^4.D_8,\mathrm{SL}_3(9).2\} & 2 \\ 
\{3^2.3^4.D_8,\mathrm{SU}_3(9).2\} & 4 \\
\{3^2.3^4.D_{10},3^2.3^4.D_{10}\} & 1 \\ 
\{3^2.3^4.D_{10},\mathrm{SL}_3(9).2\} & 6 \\ 
\{3^2.3^4.D_{10},\mathrm{SU}_3(9).2\} & 6 \\ 
\end{tabular} \end{table}

\end{document}